\documentclass[11pt,leqno,twoside]{article}
\usepackage{multirow} 
\usepackage{mathrsfs,amssymb,amsfonts,amsthm,amsmath,natbib,amsbsy}
\usepackage{dsfont,verbatim,color,fancyhdr}
\usepackage{graphicx,float} 
\usepackage[hang]{caption} 
\usepackage{enumerate} 
\usepackage[normalem]{ulem}
\usepackage[bookmarks,bookmarksnumbered,colorlinks=true,pdfstartview=FitV, linkcolor=red,citecolor=blue,urlcolor=blue]{hyperref}
\usepackage{pdfsync}

\allowdisplaybreaks

\bibpunct{\textcolor{blue}{(}}{\textcolor{blue}{)}}{;}{a}{\textcolor{blue}{,}}{\textcolor{blue}{;}}

\long\def\symbolfootnote[#1]#2{\begingroup\def\thefootnote{\hspace*{-1mm}\fnsymbol{footnote}}\footnote[#1]{#2}\endgroup}

\def \A {\mathcal{A}}
\def \B {\mathcal{B}}
\def \D {\mathscr{D}}
\def \E {\mathbb{E}}
\def \I {\mathbb{I}}
\def \N {\mathbb{N}}
\def \P {\mathbb{P}}
\def \PD {\mathrm{PD}}
\def \T {\mathcal{T}}

\def \DK {\Delta_{K}}
\def \NK {\nabla_{K}}
\def \Ni {\nabla_{\infty}}
\def \Nic {\overline{\nabla}_{\infty}}
\def \eps {\varepsilon}

\def \hat {\widehat}
\def \tilde {\widetilde}
\def \bar{\overline}
\def \e {\mathbb{E}}

\DeclareMathOperator*{\bplim}{bp\,\mbox{-}lim}
\DeclareMathOperator*{\grad}{grad}

\newtheorem{theorem}{Theorem}[section]
\newtheorem{proposition}[theorem]{Proposition} 

\newtheorem{lemma}[theorem]{Lemma} 
 
\newtheorem{definition1}[theorem]{Definition} 
 
\newtheorem{remark1}[theorem]{Remark} \newenvironment{remark}{\begin{remark1}\rm}{\hfill$\square$\end{remark1}}

\title{\vspace{-15mm}\bf  
Wright--Fisher construction of the two-parameter Poisson--Dirichlet diffusion
}
\author{
\textsc{Cristina Costantini}\\
\emph{University of Chieti-Pescara}\\[1mm]
\textsc{Pierpaolo De Blasi}\\
\emph{University of Torino and Collegio Carlo Alberto}\\[1mm]
\textsc{Stewart N. Ethier}\\
\emph{University of Utah}\\[1mm]
\textsc{Matteo Ruggiero}\\
\emph{University of Torino and Collegio Carlo Alberto}\\[1mm]
\textsc{Dario Span\`o}\\
\emph{University of Warwick}
}
\date{\today}

\fancypagestyle{plain}{
\fancyhead{}
}

\begin{document}

\maketitle
\thispagestyle{empty}

\begin{abstract}
The two-parameter Poisson--Dirichlet diffusion, introduced in 2009 by Petrov, extends the infinitely-many-neutral-alleles diffusion model, related to Kingman's one-parameter Poisson--Dirichlet distribution and to certain Fleming--Viot processes. The additional parameter has been shown to regulate the clustering structure of the population, but is yet to be fully understood in the way it governs the reproductive process. Here we shed some light on these dynamics by formulating a $K$-allele Wright--Fisher model for a population of size $N$, involving a uniform mutation pattern and a specific state-dependent migration mechanism.  Suitably scaled, this process converges in distribution to a $K$-dimensional diffusion process as $N\to\infty$.  Moreover, the descending order statistics of the $K$-dimensional diffusion converge in distribution to the two-parameter Poisson--Dirichlet diffusion as $K\to\infty$.  The choice of the migration mechanism depends on a delicate balance between reinforcement and redistributive effects.  The proof of convergence to the infinite-dimensional diffusion is nontrivial because the generators do not converge on a core.  Our strategy for overcoming this complication is to prove \textit{a priori} that in the limit there is no ``loss of mass'', i.e., that, for each limit point of the sequence of finite-dimensional diffusions (after a reordering of components by size), allele frequencies sum to one.\medskip

\medskip

\end{abstract}

%%%%%%%%%%%%%%%%%%%%%%%%%%%%%%%%%%%%%%%%%%%%

\section{Introduction}\label{intro}

The goal of this paper is to provide a discrete-time finite-population construction of the two-parameter Poisson--Dirichlet diffusion, extending an analogous construction for the well-known infinitely-many-neutral-alleles diffusion model provided in \cite{EK81}.  Introduced by \cite{P09} and henceforth called the two-parameter model, this diffusion process assumes values in the infinite-dimensional ordered simplex (sometimes also called the Kingman simplex)
\begin{equation}\label{nabla-infty-closure}
\Nic:=\bigg\{z=(z_1,z_2,\ldots)\in[0,1]^{\infty}:\,z_{1}\ge z_{2}\ge\cdots\ge0,\;\sum_{i=1}^{\infty}z_{i}\le1\bigg\}
\end{equation}
and describes the temporal evolution of the ranked frequencies of infinitely many potential alleles, observed at a single gene locus, in a given population of large but finite size.  An exhaustive review of these and other models for stochastic population dynamics can be found in \cite{F10}. Further investigations of the two-parameter model include \cite{RW09}, who provide a particle construction; \cite{FS10}, who study some path properties using Dirichlet forms; \cite{Fetal11}, who find the transition density function; \cite{RWF13}, who show that an instance of the two-parameter model arises as a normalised inverse-Gaussian diffusion conditioned on having a fixed environment; \cite{R14}, who shows that the clustering structure in the population is driven by a continuous-state branching process with immigration; \cite{E14}, who shows that, with probability one, the diffusion instantly enters the dense subset 
\begin{equation}\label{nabla-infty}
\Ni:=\bigg\{z=(z_1,z_2,\ldots)\in[0,1]^{\infty}:\,z_{1}\ge z_{2}\ge\cdots\ge0,\;\sum_{i=1}^{\infty}z_{i}=1\bigg\}
\end{equation}
and never exits; and \cite{Z15}, who simplifies the formula for the transition density and establishes an ergodic inequality.

The two-parameter model is known to be reversible \citep{P09} with respect to the two-parameter Poisson--Dirichlet distribution $\PD(\theta,\alpha)$, where $0\le\alpha<1$ and $\theta>-\alpha$ \citep{PPY92,P95,PY97}. When $\alpha=0$, the model reduces to the infinitely-many-neutral-alleles diffusion model, henceforth called the one-parameter model, with Poisson--Dirichlet reversible distribution $\PD(\theta):=\PD(\theta,0)$ \citep{K75}.  The two-parameter Poisson--Dirichlet distribution $\PD(\theta,\alpha)$ has found numerous applications in several fields: See for example \cite{B06} for fragmentation and coalescent theory; \cite{P06} for excursion theory and combinatorics; \cite{LP09} for Bayesian inference; and \cite{TJ09} for machine learning.  However, in the dynamic setting, the two-parameter  model is not as well understood as the one-parameter special case, which motivates the need for further investigation. 

One of the main differences between the $\PD(\theta)$ and $\PD(\theta,\alpha)$ distributions is the fact that the former arises as the weak limit of ranked Dirichlet frequencies \citep{K75}, whereas a similar construction is not available for the two-parameter case. In the dynamical framework, one possible construction of the one-parameter model is as the limit in distribution as $K\to\infty$ of a $K$-dimensional diffusion process of Wright--Fisher type with components rearranged in descending order. Each of these Wright--Fisher diffusions can in turn be constructed as the limit in distribution as $N\to\infty$ of a suitably scaled $K$-allele Wright--Fisher Markov chain model for a randomly mating population of size $N$ with discrete nonoverlapping generations and uniform mutation \citep{EK81}.  In contrast, an analogous construction in the case $0<\alpha<1$ has not, to the best of our knowledge, been published.  The importance of finding examples of processes with these features for the two-parameter model lies in the possibility of revealing the reproductive mechanisms acting at the level of individuals, thus providing interpretation for the roles played by the parameters $\theta$ and $\alpha$ in the dynamics of the population's allele frequencies, partially hidden or difficult to interpret in the infinite-dimensional model. In Section \ref{sec: 2parmodel} we will provide more comments on this point and on the other existing sequential constructions for the two-parameter model.

In this paper we show that the two-parameter model can be derived from a Wright--Fisher Markov chain model.  As with the one-parameter model, there are two limit operations involved.  We start with a $K$-allele Wright--Fisher model for a randomly mating population of size $N$ with discrete nonoverlapping generations, a uniform mutation pattern, and a specific state-dependent migration mechanism. It is not difficult to see that this process, suitably scaled, converges in distribution to a $K$-dimensional Wright--Fisher diffusion as $N\to\infty$.  The process obtained by applying the descending order statistics to this Wright--Fisher diffusion is itself a diffusion (i.e., the Markov property is retained), which we show converges in distribution to the two-parameter model as $K\to\infty$. 

We also show that the two-parameter Poisson--Dirichlet distribution\linebreak $\text{PD}(\theta,\alpha)$ is the weak limit of the stationary distributions of the Wright--Fisher diffusions we obtain (modified to account for the rearranging of components in descending order), by analogy to what happens in the one-parameter case, where these stationary distributions are symmetric Dirichlet distributions.

Our Wright--Fisher model includes migration and mutation.  Mutation is uniform as before but with mutation rate proportional to $\theta+\alpha$ instead of just $\theta$.  Migration, which acts first, also depends on $\alpha$ and is governed by a generalisation of the classical island model.  In that model, the frequency of allele $i$ on the island after migration (in the gametic pool) is
\begin{equation}\label{island1}
z_i^*=z_i+p_i m-z_i m,
\end{equation}
with $z_i$ being its frequency on the island prior to migration, $m$ being the migration rate, and $p_i$ being the frequency of allele $i$ in the mainland population.  We generalise this in two ways, neither of which is conventional in the population genetics literature.  First, we allow the migration rate to be allele-dependent, so that \eqref{island1} is replaced by
\begin{equation}\label{island2}
z_i^*=z_i+p_i m(z)-z_i m_i,\quad \text{where}\quad m(z)=\sum_{j=1}^K z_j m_j.
\end{equation}
Here $m_i$ is the migration rate for allele $i$ and $m(z)$ is the overall migration rate.  The second generalisation allows all parameters to be state-dependent, that is, to depend on the vector $z$ of allele frequencies on the island.  Thus, \eqref{island2} is replaced by
\begin{equation}\label{island3}
z_i^*=z_i+p_i(z) m(z)-z_i m_i(z),\quad \text{where}\quad m(z)=\sum_{j=1}^K z_j m_j(z).
\end{equation}
Here $m_i(z)$ is the migration rate for allele $i$, $m(z)$ is the overall migration rate, and $p_i(z)$ is the frequency of allele $i$ in the mainland population.  The  form of the functions $m_i(z)$ and $p_i(z)$ will be specified later on, but for now we point out only that $m_i(z)$ depends on $z_i$ alone and is a decreasing function of that variable that does not depend in $i$, and $p_i(z)>p_j(z)$ if $z_i<z_j$.  Thus, more frequent alleles on the island are less likely to emigrate (so emigration provides a reinforcement effect), and less frequent alleles on the island are more frequent on the mainland and therefore more likely to immigrate (so immigration provides a redistributive effect).

The proof of convergence in distribution of the $K$-dimensional Wright--Fisher diffusion, with components rearranged in descending order, to the two-parameter model as $K\to\infty$ is nontrivial and requires a new approach.  The difficulty arises essentially from the fact that, with $\B_K$ denoting the generator of the reordered $K$-dimensional diffusion, and $\B$ denoting the generator of the two-parameter model, $\B_K\varphi$ does not converge to $\B\varphi$ on $\Nic$ for certain $\varphi$ in the domain $\D(\B)$ of $\B$ (see Section \ref{sec: 2parmodel}).  The simplest such $\varphi$ is the so-called {\em homozygosity}, $\varphi_2(z):=\sum\nolimits_{i=1}^{\infty}z_i^2$.  At the same time, it is not possible to eliminate $\varphi_2$ from the domain of $\B$, because the resulting space of functions would not be a core for the closure of $\B$.  As a consequence, the approach followed in \cite{EK81} to study the one-parameter model fails here, and so do various other similar approaches.  A more complete discussion of these issues can be found at the beginning of Section \ref{sec: convergence to full model}. 

Here we take the martingale problem approach, i.e., we view the reordered $K$-dimensional diffusion as the solution of the martingale problem for $\B_K$, and, as is usual in this approach, try to carry out three steps:  (i) Show that the sequence of finite-dimensional diffusions is relatively compact; (ii) Show that each of its limit points is a solution to the martingale problem for $\B$;  (iii) Show that the martingale problem for $\B$ has a unique solution.  As may be expected, the difficulty described above shows up in this approach as well: If the domain of $\B$ includes $\varphi_2$, then it is not clear that the limit martingale property will hold for the pair $(\varphi_2,\B\varphi_2)$.  On the other hand, if $\varphi_2$ is excluded from the domain of $\B$, then the martingale problem for $\B$ may have more than one solution.

However, in the martingale problem framework we are able to overcome the difficulty by proving \textit{a priori} that, for any limit point $Z$ of the sequence of finite-dimensional diffusions, with probability one, for almost all $t\geq 0$, $Z(t)$ belongs to $\Ni$ (cf.~\eqref{nabla-infty}).  The argument employed in this proof was inspired by the proof of Theorem 2.6 of \cite{EK81} and relies on a double limit, taken in the appropriate order.  When restricted to $\Ni$, $\B_K\varphi_2$ does converge to $\B\varphi_2$, and this yields that the limit martingale property carries over to $(\varphi_2,\B\varphi_2)$, and thus that the limit martingale problem has a unique solution.

In the one-parameter case, it is possible to also formulate a Wright--Fisher model with infinitely many alleles and obtain the limit process as $N\to\infty$ ($K$ is already $\infty$); see \cite{EK81}, Theorem 3.3.  That theorem requires some rather delicate estimates and we were unsuccessful in trying to extend it to the two-parameter setting.  

The paper is organised as follows. In Section \ref{sec: 2parmodel} the two-parameter model is recalled.  Section \ref{sec: WF chain} provides the construction of the $K$-allele Wright--Fisher Markov chain for a population of size $N$.  In Section \ref{sec: K-diffusion} the Wright--Fisher chain, scaled appropriately, is shown to converge in distribution to a $K$-dimensional Wright--Fisher diffusion as $N\to\infty$.  Then, in Section \ref{sec: convergence to full model}, the $K$-dimensional diffusion, with coordinates rearranged in descending order, is shown to converge to the two-parameter model as $K\to\infty$.  In Section \ref{sec: conv of stat dist} analogous results are proved for the stationary distributions.  Section \ref{sec:7} concludes by highlighting a slightly simpler formulation, obtained under the assumption that $\theta\ge0$, which allows us to separate the roles of $\theta$ and $\alpha$ in driving the population dynamics.

%%%%%%%%%%%%%%%%%%%%%%%%%%%%%%%%%%%%%%%%%%%

\section{The two-parameter model}
\label{sec: 2parmodel}

The two-parameter model was introduced by \cite{P09}.  As with its one-parameter counterpart, characterised in \cite{EK81}, it describes the temporal evolution of infinitely many allele frequencies.  A natural state space for the process is $\Ni$, defined in \eqref{nabla-infty}.  However, the closure of $\Ni$ (in the product topology on $[0,1]^\infty$), namely $\Nic$, defined in \eqref{nabla-infty-closure}, is compact and therefore more convenient as a state space.  Consider, for parameters $0\le\alpha<1$ and $\theta>-\alpha$, the second-order differential operator $\B$ defined as follows. The domain of $\B$ is 
\begin{equation}\label{domain B}
\D(\B):=\text{subalgebra of } C(\Nic)\text{ generated by }\varphi_1,\varphi_2,\varphi_3,\ldots,
\end{equation} 
where $\varphi_1\equiv1$ and, for $m=2,3,\ldots$, $\varphi_m$ is defined by
\begin{equation}\label{varphi_m}
\varphi_{m}(z):=\sum_{i=1}^\infty z_{i}^{m}.
\end{equation}
For $\varphi \in \D(\B)$, $\B \varphi$ is the continuous extension to $\Nic$ of 
\begin{equation}\label{operator inf-inf}
\B \varphi (z):=\frac{1}{2}\sum_{i,j=1}^{\infty}z_{i}(\delta_{ij}-z_{j})\frac{\partial^{2} \varphi(z)}{\partial z_{i}\,\partial z_{j}}-\frac{1}{2}\sum_{i=1}^{\infty}(\theta z_{i}+\alpha)\frac{\partial \varphi(z)}{\partial z_{i}},\quad z \in \Ni ,
\end{equation} 
with $\delta_{ij}$ the Kronecker delta.  For example,
\begin{equation}\label{Bvarphi2}
\B \varphi_2 (z):=1-\alpha-(1+\theta)\varphi_2(z),\quad z \in \Nic.
\end{equation} 
As shown by \cite{P09}, the closure of $\B$ generates a Feller semigroup on $C(\Nic)$, which characterises the finite-dimensional distributions of the two-parameter model $Z$, and the sample paths of $Z$ belong to $C_{\Nic}[0,\infty)$ with probability one.  Recently \cite{E14} proved that, for an arbitrary initial distribution $\nu\in\mathscr{P}(\Nic)$, we have
\begin{equation*}
\P(Z(t)\in\Ni \text{ for every }t>0)=1,
\end{equation*} 
that is $\Nic-\Ni$ acts as an entrance boundary (note however that technically it is not a boundary because $\Ni$ has no interior).  In particular, if $\nu(\Ni)=1$, then the sample paths of $Z$ belong to  $C_{\Ni}[0,\infty)$ with probability one.

The diffusion coefficients in the first term of \eqref{operator inf-inf} describe the instantaneous covariance, related to the allelic sampling, also called \emph{random genetic drift}.  The interpretation of the drift coefficients in the second term of \eqref{operator inf-inf} is not as clear, and is the object of primary interest in this paper.  
It is worth noting that the one-parameter model, obtained by setting $\alpha=0$ in \eqref{operator inf-inf}, admits the following two interpretations. 

First, the one-parameter model, also known as the \textit{unlabelled} infinitely-many-neutral-alleles diffusion model, has a more informative \textit{labelled} version, namely the Fleming--Viot process in $\mathscr{P}(S)$ (the set of Borel probability measures on the compact metric space $S$ with the topology of weak convergence) with mutation operator
\begin{equation*}
Ag(x):=\frac{1}{2}\theta\int_S(g(\xi)-g(x))\,\nu_0(d\xi),
\end{equation*}
where $\nu_0\in\mathscr{P}(S)$ is nonatomic.  The unlabelled model is a transformation of the labelled one. The transformation takes $\mu\in\mathscr{P}(S)$ to $z\in\Nic$, where $z$ is the vector of descending order statistics of the sizes of the atoms of $\mu$. See \cite{EK93}.

The second interpretation is as the limit in distribution of a $K$-allele Wright--Fisher diffusion, with components rearranged in descending order,  as $K\to\infty$, where the rate of a mutation from allele $i$ to allele $j$ is proportional to $\theta$. See \cite{EK81}. 

As a result of these correspondences, $\theta$ is usually interpreted as the rate at which mutations occur.  Similar interpretations for the two-parameter model, however, are not available:  First, the existence of a Fleming--Viot process whose unlabelled version is the two-parameter model is an open problem (posed in \citealp{F10}).  Second, for $0<\alpha<1$, a Kingman-type result expressing PD($\theta,\alpha$) as the limit in distribution of a sequence of finite-dimensional random vectors is not available; hence it does not offer a guide for a Wright--Fisher construction, as in the one-parameter case.  Consequently, the interpretation of $\alpha$ cannot be deduced from existing work.  The role of $\alpha$ has been associated rather indirectly to mutation in a particle construction of the two-parameter model, given in \cite{RW09}, where $\theta$ and $\alpha$ jointly regulate births from the same distribution. They propose a Moran-type process for the evolution of $N$ individuals, whereby at exponential times a randomly chosen individual is removed from the population after either giving a simple birth, with the offspring inheriting the parent's type, or giving a birth with mutation, with the offspring being of a type not previously observed. The probabilities of these events are regulated by the weights of Pitman's generalisation of the Blackwell-MacQueen P\'olya urn scheme \cite{P95,P96}. In particular, with probability proportional to $\theta+\alpha k$, where $k$ is the current number of distinct types present in the population, a birth with mutation occurs, whereas with probability proportional to $n_{j}-\alpha k$, where $n_{j}$ is the current number of type-$j$ individuals, a simple birth of type $j$ occurs. The original sequential construction in \cite{P09} instead relates to a discrete Markov chain on the space of partitions of $\{1,\ldots,N\}$ and offers no insight into the role of $\alpha$ at the reproduction level. 
Both these constructions feature overlapping generations and fall into the infinitely-many-types setting, in the sense that they both allow the possibility of new types appearing in the population chosen from an uncountable genetic pool. 

Here, instead, we are interested in a construction of the two-parameter model by means of a classical Wright--Fisher Markov chain, with non overlapping generations and finitely many types, since this would reveal details about how the reproduction acts at the individual level, which an inspection of $\B$ does not.  As an illustration of this aspect, consider the construction of the one-parameter model via a Wright--Fisher Markov chain with $K$ alleles in a population of size $N$.  If $z=(z_1,\ldots,z_K)$ is the vector of allele frequencies prior to mutation, the frequency of allele $i$ individuals after mutation is $z_{i}(1-\sum\nolimits_{j:j\ne i}u_{ij})+\sum\nolimits_{j:j\ne i}z_{j}u_{ji}$, where 
\begin{equation}\label{mutation probability}
u_{ij}:=\frac{\theta}{2N(K-1)}, \quad j\ne i,
\end{equation} 
is the proportion of individuals of allele $i$ that mutate to allele $j$, for sufficiently large $N$.  It can be easily seen that the expected change of $z_{i}$, multiplied by $N$, is given by the drift coefficient
\begin{equation}\label{theta drift}
\frac{1}{2}\bigg[\frac{\theta}{K-1}(1-z_{i})-\theta z_{i}\bigg],
\end{equation} 
which converges to $-(1/2)\theta z_{i}$ when $K\to\infty$. See \cite{EK81} for more details. This construction provides insight into the role of $\theta$ in the mutation process, only partially readable from \eqref{operator inf-inf} with $\alpha=0$;  it is indeed by inspection of \eqref{mutation probability} that one can see that the probability of an individual mutation is inversely proportional to the population size and the mutant type distribution is uniform on the other $K-1$ alleles; the rate $\theta$ determines how often the mutation events occur.  

Here we seek a similar insight, at the same level of magnification, on the action of $\alpha$ in the two-parameter model. In this case, the drift coefficients in \eqref{operator inf-inf} are $-\frac{1}{2}(\theta z_i+\alpha)$.  The key observation for the following development is to think of them as
\begin{equation}\label{limit drift}
-\frac{1}{2}(\theta+\alpha)z_i-\frac{1}{2}\alpha(1-z_i),
\end{equation}
the first term corresponding to mutation and the second term to migration. The first term is the limit as $K\to\infty$ of the analogue of \eqref{theta drift}, namely
\begin{equation*}
\frac{1}{2}\bigg[\frac{\theta+\alpha}{K-1}(1-z_{i})-(\theta+\alpha) z_{i}\bigg], 
\end{equation*}
while the second term should be the limit of the migration terms in the $K$-allele drift coefficients.

%%%%%%%%%%%%%%%%%%%%%%%%%%%%%%%%%%%

\section{A Wright--Fisher model with state-dependent migration}\label{sec: WF chain}%\break

Consider a population of $N$ individuals, and let the maximum number of alleles in the population be $K\ge 2$. The population size is assumed to be constant and generations are nonoverlapping.  Denote by $z_{i}$ the relative frequency of allele $i$ in the current generation at the selected locus.  We assume the presence of migration and mutation, as discussed in Section \ref{intro}.  The state space is
\begin{equation}\label{DK}
\DK:=\bigg\{z=(z_1,\ldots,z_K)\in[0,1]^K:\; z_1\geq 0,\ldots,z_K\ge0,\, \sum_{i=1}^K z_i=1  \bigg\}
\end{equation} 
or, more precisely, 
\begin{equation}\label{DK-discrete}
\Delta_K^N:=\{z=(z_1,\ldots,z_K)\in\Delta_K: Nz\in\mathbb{Z}^K\}.
\end{equation}
The frequency of allele $i$ after migration (in the gametic pool) is 
\begin{equation}\label{migration}
z_i^*=z_i+p_i(z) m(z)-z_i m_i(z),\quad \text{where}\quad m(z)=\sum_{j=1}^K z_j m_j(z),
\end{equation}
as discussed in the Introduction (see \eqref{island3}).  With $u_{ij}$ denoting the proportion of individuals of allele $i$ that mutate to allele $j$, the frequency of allele $i$ after mutation (in the gametic pool) is
\begin{equation}\label{mutation}
z_{i}^{**}:=z_i^*+\sum_{j:j\ne i}z_{j}^{*}u_{ji}-z_i^*\sum_{j:j\ne i}u_{ij}.
\end{equation} 
Finally, random genetic drift is modelled by multinomial sampling, which amounts to assuming that each individual of the next generation chooses its parent at random from the current generation.  Then the next generation's allele frequencies $z_1',\ldots,z_K'$ are formed according to the rule 
\begin{equation}\label{reproduction mechanism}
z'\mid z\sim N^{-1}\text{multinomial}(N,z^{**}_{1},\ldots,z^{**}_{K}),
\end{equation}
i.e., $N z'$ has a multinomial distribution with sample size $N$ and cell probabilities $(z^{**}_{1},\ldots,z^{**}_{K})$. This is the classic Wright--Fisher model with migration and mutation in the state space $\Delta_K^N$, and without migration it corresponds to eq.~(2.2) in \cite{EK81}. For a more complete description of the Wright--Fisher model and its underlying assumptions, we refer the reader to Section 9.9 of \cite{N92}, but with selection replaced by migration.  (In Nagylaki's notation, (9.158), (9.155), and (9.146) are replaced respectively by $p_i^*=p_i+\hat p_i(\textbf{p})m(\textbf{p})-p_i m_i(\textbf{p})$ with $m(\textbf{p})=\sum_j p_j m_j(\textbf{p})$, $\tilde P_{ij}^*=p_i^*p_j^*$, and $P_{ij}^*=(2-\delta_{ij})p_i^*p_j^*$. Finally, our $N$ is Nagylaki's $2N$.)

We turn now to specifying the migration and mutation in sufficient detail to derive a $K$-allele diffusion approximation.
Consider parameter values of $0\le \alpha<1$ and $\theta>-\alpha$ (the case $\theta\ge0$, which allows a simplification, is treated separately in Section \ref{sec:7}). We assume that the migration rates are given by
\begin{equation}\label{migration rates}
m_i(z):=\frac{\alpha r_{i}(z)}{2N},\quad i=1,\ldots,K,
\end{equation} 
and that the mutation rates $u_{ij}$ are given by
\begin{equation}\label{mutation rates}
u_{ij}:=\frac{\theta+\alpha}{2N(K-1)}, \quad j\ne i,
\end{equation} 
for sufficiently large $N$ (cf.~\eqref{mutation probability}). The functions $p_i$ in \eqref{migration} and $r_i$ in \eqref{migration rates}, defined on $\DK$ for $i=1,\ldots,K$, are assumed to satisfy the following properties:  $(p_1,\ldots,p_K)$ is a $C^4$ map of $\DK$ into $\DK$ and is symmetric in the sense that, for every permutation $\sigma$ of $\{1,2,\ldots,K\}$,
$$
p_i(z_{\sigma(1)},\ldots,z_{\sigma(K)})=p_{\sigma(i)}(z),\quad i=1,\ldots,K,\; z\in\DK;
$$
$p_i(z)>p_j(z)$ if $z_i<z_j$ for all $z\in\DK$ and $i\ne j$; $r_i(z)=r(z_i)$ for $i=1,\ldots,K$ and $z\in\DK$, where $r:[0,1]\mapsto[0,\infty)$ is $C^4$ and is decreasing.  

In Section \ref{sec: convergence to full model} we will be more specific as to the form of $p_i(z)$ and $r_i(z)$ (see \eqref{condizioni r} and \eqref{condizioni p} below).  In Section \ref{sec:7} we will give a simpler formulation of $p_i(z)$ and $r_i(z)$ in the special case $\theta\geq 0$.  Here and later, for notational simplicity, we suppress the dependence on $K$ of the defined quantities, whenever this does not create confusion.  

To summarise, our Markov chain $Z_K^N(\cdot)=\{Z_K^N(\tau),\,\tau=0,1,\ldots\}$ has state space $\Delta_K^N$ (see \eqref{DK-discrete}) and its transition probabilities are specified by \eqref{migration}--\eqref{mutation rates}.

From \eqref{migration} and \eqref{mutation}, we can write the frequency of allele $i$ (in the gametic pool) at reproductive age in terms of the allele frequencies before the action of migration and mutation as
\begin{equation}\label{z^*}
z^{**}_{i}=z_{i}+N^{-1}b_i(z)+o(N^{-1}),
\end{equation} 
uniformly in $z\in\Delta_K^N$, where, in view of the rescaling, we have isolated the relevant drift term for the $i$th component, namely
\begin{equation}\label{drift b_i}
\quad b_i(z):=\,\frac{1}{2}\bigg[\frac{\theta+\alpha}{K-1}(1-z_{i})-(\theta+\alpha) z_{i}+\alpha p_i(z) \sum_{j=1}^{K}z_j r_{j}(z) -\alpha z_i r_{i}(z)\bigg].
\end{equation}
For later use, note that 
\begin{equation}\label{boundary condition 1}
b_i(z)\ge0\text{ if }z_i=0,\quad i=1,2,\ldots,K,
\end{equation}
and
\begin{equation}\label{boundary condition 2}
\sum_{i=1}^K b_i(z)=0,\quad z\in\Delta_K.
\end{equation}

%%%%%%%%%%%%%%%%%%%%%%%%%%%%%%%%%%

\section{Diffusion approximation with \texorpdfstring{$K$}{K} alleles}\label{sec: K-diffusion}

Recall \eqref{DK}, and define the second-order differential operator
\begin{equation}\label{operator A_K}
\A_{K}:=\frac{1}{2}\sum_{i,j=1}^{K}a_{ij}(z)\frac{\partial^{2}}{\partial z_{i}\partial z_{j}}+\sum_{i=1}^{K}b_i(z)\frac{\partial}{\partial z_{i}},\qquad \D(\A_{K})=C^{2}(\Delta_K),
\end{equation}
with
\begin{equation}\label{diff a}
a_{ij}(z):=z_{i}(\delta_{ij}-z_{j})
\end{equation}
and $b_i(z)$, which of course depends on $K$, as in \eqref{drift b_i}. Here
\begin{equation*}
C^{2}(\Delta_K):=\{f\in C(\DK):\  \exists \tilde f\in C^2(\mathds{R}^K)\text{ such that }\tilde f|_{\Delta_K}=f\},
\end{equation*} 
and the choice of the extension $\tilde f$ to which the partial derivatives are applied does not matter.  Let $C(\DK)$ be endowed with the supremum norm.  The following result states that $\A_{K}$ characterises a Feller diffusion on $\DK$.

\begin{proposition}\label{thm: K-diffusion}
Let $\A_K$ be as in \eqref{operator A_K}--\eqref{diff a} and \eqref{drift b_i}. Then the closure in $C(\DK)$ of $\A_K$ is single-valued and generates a Feller semigroup $\{\T_K(t)\}$ on $C(\DK)$. For each $\nu_K\in\mathscr{P}(\Delta_K)$, there exists a strong Markov process $Z_K(\cdot)=\{Z_K(t),\; t\geq 0\}$, with initial distribution $\nu_{K}$, such that
\begin{equation*}
\mathbb{E}(f(Z_K(t+s))\mid Z_K(u),\, u\le s)=\mathcal{T}_{K}(t)f(Z_K(s)),
\quad f\in C(\DK),\; s,t\ge0.
\end{equation*}  
Furthermore, 
\begin{equation*}
\mathbb{P}\{Z_K(\cdot)\in C_{\Delta_K}[0,\infty)\}=1.
\end{equation*} 
\end{proposition}

\begin{proof}
Noting that $b_1,\ldots,b_K \in C^{4}(\Delta_{K})$, the first assertion follows from \cite{E76} and \cite{S78}, using \eqref{boundary condition 1} and \eqref{boundary condition 2}.  The second assertion follows from Theorem 4.2.7 in \cite{EK86}. Note that for every $z^{0}\in\Delta_K$ and $\eps>0$ there exists $f\in\D(\A_{K})$ such that
\begin{equation*}
\sup_{z\in B(z^{0},\eps)^c}f(z)<f(z^{0})=\|f\|\quad \text{and} \quad \A_{K}f(z^{0})=0,
\end{equation*} 
where $B(z^{0},\eps)$ is the ball of radius $\eps$ centred at $z^{0}$. Take for example $f(z):=2-\sum_{i=1}^{K}(z_{i}-z_i^{0})^{4}$.  Then the third assertion follows from Proposition 4.2.9 and Remark 4.2.10 in \cite{EK86}.
\end{proof}

The diffusion of Proposition \ref{thm: K-diffusion} is a good approximation, in the sense of the limit in distribution as the population size tends to infinity, of a suitably rescaled version of the Wright--Fisher Markov chain described in Section \ref{sec: WF chain}. This is formalised by the next theorem. Here and later $\Rightarrow$ denotes convergence in distribution (or weak convergence) and $D_{\DK}[0,\infty)$ denotes the space of \emph{c\`adl\`ag} sample paths in $\DK$ with the Skorokhod topology.

\begin{theorem}\label{thm: chain convergence}
Let $\{Z_K^N(\tau),\; \tau=0,1,\ldots\}$ be the $\Delta_K^N$-valued Markov chain with one-step transitions as in \eqref{migration}--\eqref{mutation rates}, let $Z_K$ be the Feller diffusion of Proposition \ref{thm: K-diffusion}. If $Z_K^N(0)\Rightarrow Z_K(0)$, then 
\begin{equation*}
Z_K^N(\lfloor N\cdot\rfloor)\Rightarrow Z_K(\cdot)\text{ in }D_{\DK}[0,\infty)
\end{equation*}
as $N\to\infty$.
\end{theorem}

\begin{proof}
From \eqref{reproduction mechanism} and \eqref{z^*}, letting $\E_{z}(\cdot):=\E(\cdot\mid z)$ and similarly for $\P_{z}$, we have that
\begin{equation*}
\E_{z}[z_i'-z_i]=\E_{z}[z_{i}'-z_{i}^{**}]+z_{i}^{**}-z_{i}=N^{-1}b_i(z)+o(N^{-1})
\end{equation*}
and
\begin{equation*}
\E_{z}[(z_i'-z_i)(z_j'-z_j)]=N^{-1}z_{i}^{**}(\delta_{ij}-z_{j}^{**})+o(N^{-1})=N^{-1}a_{ij}(z)+o(N^{-1}),
\end{equation*}
uniformly in $z$.  Furthermore, it can be easily seen that $\E_{z}[(z_i'-z_i)^{4}]=o(N^{-1})$, so that  Chebyshev's inequality implies Dynkin's condition for the continuity of paths of the limit process, that is, $\mathbb{P}_{z}(|z_i'-z_i|>\delta)=o(N^{-1})$ for every $\delta>0$.  Again, these estimates are uniform in $z$.  Denote by $\mathcal{T}_K^N$ the semigroup operator associated to the Markov chain $Z_K^N(\cdot)$ and by $I$ the identity operator. Then a Taylor expansion, together with the above expressions, yields, for every $f\in C^{2}(\DK)$,
\begin{align*}
(\mathcal{T}_K^N-I)f(z)
&=\E_{z}[f(z')-f(z)]\\
&=\E_z\bigg[\sum_{i=1}^{K}(z_i'-z_i)f_{z_i}(z)+\frac{1}{2}\sum_{i,j=1}^{K}(z_i'-z_i)(z_j'-z_j)f_{z_iz_j}(z)\\
&\qquad\quad{}+\int_0^1(1-t)\sum_{i,j=1}^K(z_i'-z_i)(z_j'-z_j)\\
&\qquad\qquad\quad \times [f_{z_iz_j}(z+t(z'-z))-f_{z_iz_j}(z)]\,dt\bigg]\\
&=\frac{1}{N}\sum_{i=1}^{K}b_i(z) f_{z_i}(z)+\frac{1}{2N}\sum_{i,j=1}^{K}a_{ij}(z)f_{z_iz_j}(z)+o\bigg(\frac{1}{N}\bigg),
\end{align*}
uniformly in $z$, 
where the $o(N^{-1})$ term above is due to 
\begin{align*}
\bigg|&\E_{z}\bigg[\int_0^1(1-t)\sum_{i,j=1}^K(z_i'-z_i)(z_j'-z_j)[f_{z_iz_j}(z+t(z'-z))-f_{z_iz_j}(z)]\,dt\bigg]\bigg|\\
&\quad{}\le \E_{z}\bigg[\frac12\sum_{i,j=1}^K|z_i'-z_i||z_j'-z_j|\,2\|f_{z_iz_j}\|; |z'-z|>\delta\bigg]\\
&\quad\qquad{}+\E_{z}\bigg[\frac12\sum_{i,j=1}^K|z_i'-z_i||z_j'-z_j|\,\omega(f_{z_iz_j},\delta); |z'-z|\le\delta\bigg]\\
&\quad{}\le \sum_{i,j=1}^K\|f_{z_iz_j}\|\,\P_z(|z'-z|>\delta)\\
&\quad\qquad{}+\frac12\sum_{i,j=1}^K\E_{z}[(z_i'-z_i)^2]^{1/2}\E_{z}[(z_j'-z_j)^2]^{1/2}\,\omega(f_{z_iz_j},\delta)\\
&\quad=o(N^{-1})+O(N^{-1})\max_{i,j}\omega(f_{z_iz_j},\delta),
\end{align*}
$\omega(g,\delta):=\sup_{|z'-z|\le\delta}|g(z')-g(z)|$ being the modulus of continuity of the function $g$.
It follows that, for every $f\in C^{2}(\DK)$,
\begin{equation}\label{WF-conv}
\|N(\mathcal{T}_K^N-I)f-\A_{K}f\|\to0
\end{equation} 
as $N\to\infty$, where $\A_{K}$ is as in \eqref{operator A_K}.  An application of Theorems 1.6.5 and 4.2.6 in \cite{EK86} implies the statement of the theorem.
\end{proof}

Having justified our first limit operation, we now apply the descending order statistics to our limit Wright--Fisher diffusion $Z_K(\cdot)$.  First, we define the continuous map $\rho_{K}:\Delta_K\mapsto \Ni$ by
\begin{equation*}\label{rho-n}
\rho_{K}(z):=(z_{(1)},\ldots,z_{(K)},0,0,\ldots),
\end{equation*} 
where $z_{(1)}\ge z_{(2)}\ge \cdots \ge z_{(K)}$ are the descending order statistics of the coordinates of $z\in\Delta_K$.  We will show in the next section that, with suitable definitions of $p_i(z)$ and $r_i(z)$ and assuming convergence of the initial distributions, $\rho_K(Z_K(\cdot))\Rightarrow Z(\cdot)$ as $K\to\infty$, with $Z(\cdot)$ denoting the two-parameter model in $\Nic$.

Here we simply observe that $\rho_K(Z_K(\cdot))$ is Markovian despite the fact that $\rho_K$ is not one-to-one.  The state space of $\rho_K(Z_K(\cdot))$ is
\begin{equation*}
\NK:=\{z\in\Ni:\; z_{K+1}=0\}\subset\Nic
\end{equation*} 
and its generator $\B_K$ is given by
\begin{equation}\label{operator B_K}
\B_{K}:=\frac{1}{2}\sum_{i,j=1}^{K}a_{ij}(z)\frac{\partial^{2}}{\partial z_{i}\partial z_{j}}+\sum_{i=1}^{K}b_i(z)\frac{\partial}{\partial z_{i}},
\end{equation} 
ostensibly the same as $\A_K$ in \eqref{operator A_K}--\eqref{diff a} and \eqref{drift b_i}, except that now $z\in\NK$ instead of $z\in\DK$.  In addition, 
\begin{equation}\label{domain B_K}
\mathscr{D}(\B_K):=\{f\in C^2(\NK): f\circ\rho_K\in C^2(\DK)\}.
\end{equation}
Hidden in this definition are certain implicit boundary conditions needed to preserve the inequalities $z_1\ge z_2\ge \cdots\ge z_K$ (see \cite{EK81} for more details). The following result generalises Proposition 2.4 of \cite{EK81}.

\begin{proposition}\label{rhoK-ZK}
The closure in $C(\NK)$ of the operator $\B_K$ defined by \eqref{operator B_K}--\eqref{domain B_K}, \eqref{diff a} and \eqref{drift b_i} is single-valued and generates a Feller semigroup $\{\mathcal{U}_K(t)\}$ on $C(\NK)$.  Given $\nu\in\mathscr{P}(\DK)$, let $Z_K(\cdot)$ be as in Proposition \ref{thm: K-diffusion}.  Then $\rho_K(Z_K(\cdot))$ is a strong Markov process corresponding to $\{\mathcal{U}_K(t)\}$ with initial distribution $\nu_K\circ\rho_K^{-1}$ and almost all sample paths in $C_{\NK}[0,\infty)$.
\end{proposition}

\begin{proof}
The proof is exactly as in the cited paper, the key observation being that, for every permutation $\sigma$ of $\{1,2,\ldots,K\}$, 
\begin{equation*}
b_i(z_{\sigma(1)},\ldots,z_{\sigma(K)})=b_{\sigma(i)}(z), \quad z\in\DK,\; i=1,2,\ldots,K.
\end{equation*}
As a byproduct of this, we find that, if $f\in\D(\B_K)$, 
\begin{equation}\label{A vs B}
(\B_K f)\circ\rho_K =\A_K(f\circ\rho_K) \text{ on }\DK.
\end{equation}
\end{proof}

%%%%%%%%%%%%%%%%%%%%%%%%%%%%%%%%%%%

\section{Convergence to the infinite-dimensional diffusion}\label{sec: convergence to full model}

We now turn to our second limit operation, namely the convergence of the reordered Wright--Fisher diffusion $\rho_K(Z_K(\cdot))$ to the two-parameter model, i.e.,  the $\Nic$-valued diffusion process with generator $\B$ introduced in Section \ref{sec: 2parmodel}. To this end we will specify explicitly the functions $p_i$ and $r_i$ that determine the migration mechanism and provide some probabilistic interpretation of our choice, but the results of this section hold more generally (see Remark \ref{gen-immig}).

The drift coefficients of $\B$ are $-\frac{1}{2}(\theta z_i+\alpha)$, which we rewrite as in \eqref{limit drift}, while those of $\B_K$ are given by \eqref{drift b_i}.  In view of the comments at the end of Section \ref{sec: 2parmodel},  the functions $p_i$ and $r_i$ should satisfy
\begin{equation}\label{alpha drift condition}
-\frac{1}{2}\alpha(1-z_i)=\lim_{K\to\infty}\frac{1}{2} \bigg[\alpha p_i(z) \sum_{j=1}^{K}z_j r_{j}(z) - \alpha z_i r_{i}(z)\bigg].
\end{equation}
One way to achieve this is to take $r_i(z)=(1-z_i)/z_i$ and $p_i(z)=o(1/K)$.  However, this is problematic for two reasons. First, $r_i$ is unbounded; second, requiring $p_i(z)=o(1/K)$ uniformly in $i$ and $z$ is inconsistent with $\sum_{i=1}^K p_i(z)=1$.  We can address both issues by instead defining
\begin{equation}\label{condizioni r}
r_{i}(z):=\begin{cases}(1-z_i)[1-(1-z_{i})^K]/z_i&\text{if $z_i>0$}\\K&\text{if $z_i=0$}\end{cases}
\end{equation} 
and 
\begin{equation}\label{condizioni p}
p_{i}(z):=\frac{(1-z_{i})^K}{\sum\nolimits_{l=1}^{K}(1-z_l)^K}.
\end{equation} 

An alternative formulation is in terms of the following system of Bernoulli trials parameterised by the current state $z$.  Let the  array $\zeta=(\zeta_{ij})_{i,j=1,\ldots,K}$ be such that, along row $i$, $\zeta_{i1},\ldots,\zeta_{iK}$ are i.i.d.\ Bernoulli($z_i$), for $i=1,2,\ldots,K$.  With $G_i$ being the number of failures in row $i$ before the first success, 
$$
r_i(z)=\sum_{k=1}^K(1-z_i)^k=\sum_{k=1}^K\P(G_i\ge k)=\E[G_i].
$$
Furthermore, $p_i(z)$ is proportional to the probability of observing no successes in row $i$.  Incidentally, $p_i(z)$ has also a direct probabilistic interpretation via Bayes's theorem.  Let $I$ be a row of the array chosen uniformly at random.  Then $p_i(z)$ is the probability of choosing row $i$ given that we observe all failures along the row, that is,
\begin{equation*}
p_i(z)=\P\bigg\{I=i\;\bigg|\,\sum_{j=1}^{K} \zeta_{Ij}=0\bigg\}.
\end{equation*}

Let $Z$ be the two-parameter model. In order to prove that $\rho_K(Z_K(\cdot))\Rightarrow Z(\cdot )$ as $K\to\infty$, the usual argument is to show that
\begin{equation}\label{conv of generators}
\|\B_K\eta_K\varphi -\eta_K\B\varphi\|\to 0 \text{ as }K\to\infty,
\end{equation}
where $\eta_K:C(\Nic)\mapsto C(\NK)$ is given by the restriction $\eta_K\varphi =\varphi |_{\NK}$, and $\varphi\in\D(\B)$ is given by
\[
\varphi =\varphi_{m_1}\cdots\varphi_{m_l},\qquad m_1,\ldots ,m_l\in \{2,3,\ldots \},\quad l\in\N.
\]
(Notice that $\eta_K$ maps $\D(\B)$ into $\D(\B_K)$.)

Unfortunately, despite the fact that \eqref{alpha drift condition} holds with this choice, \eqref{conv of generators} fails if one or more of the subscripts $m_1,\ldots ,m_l$ is equal to 2. Similarly to what was done in \cite{EK81} in the proof of Theorem 2.6, we can enlarge the domain of $\B$ to the algebra generated by 1 and the functions $\varphi_m$ defined by $\eqref{varphi_m}$ for all real $m\ge 2$ (not just integers).  Then \eqref{conv of generators} holds for 
\[
\varphi =\varphi_{m_1}\cdots\varphi_{m_l},\qquad m_1,\ldots,m_l>2,\quad l\in \N.
\]
For example, if $\varphi =\varphi_{2+\eps}$ for $0<\eps<1$, then $\|\B_K\eta_K\varphi -\eta_K \B\varphi\|=O(K^{-\eps})$.  This would suffice if we could show that 
\begin{equation*}\label{domain B0}
\D_0(\B):=\text{subalgebra of }C(\Nic)\text{ generated by }1\text{ and }\varphi_m,\; m\in (2,\infty),
\end{equation*}
is a core for the closure of $\B$ (cf.~\cite{EK86}, Section 1.3). This also appears to fail.  In fact, this algebra is not even a core in the bounded-pointwise sense, as 
\[
\bplim_{\eps\to0+}\B\varphi_{2+\eps}(z)=(1-\alpha)\sum_{i=1}^\infty z_i-(1+\theta)\varphi_2(z),
\]
which is not equal to \eqref{Bvarphi2} except on $\Ni$. 

As mentioned in Section \ref{sec: 2parmodel}, recently \cite{E14} proved that, for any initial distribution $\nu$ concentrated on $\Ni$, the paths of $Z$ belong to $C_{\Ni}[0,\infty )$ with probability one. In view of this result and of the above discussion, one might think of taking $\Ni$ as state space, rather than $\Nic$. But $\Ni$ is not compact, therefore the usual sufficient conditions for convergence in distribution include, besides \eqref{conv of generators}, the following compact containment condition: For every $\eps ,T>0$ there exists a compact set $\Gamma_{\eps ,T}$ such that 
\begin{equation}\label{compact-containment}
\inf_K\P\left(\rho_K(Z_K(t))\in\Gamma_{\eps,T},\;\forall t\leq T\right)\geq 1-\eps.
\end{equation}
Notice that, since $\Ni$ is not a complete metric space, convergence might hold without the compact containment condition (see, for example, \cite{B68}, Theorems 6.1 and 6.2).  In any case, \eqref{compact-containment} is not easy to prove and we have not pursued this approach. 

A further alternative strategy would be to show that, for every $\psi\in\D(\B)$, there exists a sequence $\{\psi_K\}\subset\D_0(\B)$ such that 
\begin{equation*}
\|\eta_K\psi_K-\eta_K\psi\|\to 0\quad\text{{\rm and}}\quad\|\B_K\eta_K\psi_K-\eta_K\B\psi\|\to 0,
\end{equation*}
as $K\to\infty$, so that $\{(\psi ,\B\psi ):\,\psi\in\D(\B)\}$ belongs to the extended limit of $\B_K$ (cf.~Definition 1.4.3 of \cite{EK86}). Then Theorem 1.6.1 of \cite{EK86} would yield 
\begin{equation*}\label{eq:Bn-convergence}
\|\mathcal{T}_K(t)\eta_K\varphi -\eta_K\mathcal{T}(t)\varphi\|\to 0,\quad\varphi\in C(\Nic),\; t\ge0,
\end{equation*}
where $\{\mathcal{T}(t)\}$ is the Feller semigroup on $C(\Nic)$ whose generator is the closure of $\B$.  Even this strategy seems not to be viable.  

Having considered each of these routes, we have turned to the martingale problem approach. In this approach, $\rho_K(Z_K(\cdot ))$ is viewed as a solution to the martingale problem for $\B_K$ (in fact the unique solution).  The usual procedure consists of three steps: (i) Show that $\{\rho_K(Z_K(\cdot))\}$ is relatively compact.  (ii) Show that each of its limit points is a solution to the martingale problem for $\B$.  (iii) Show that the martingale problem for $\B$ has a unique solution. 

However in the present setup it is not clear how to carry out the second and third steps. In fact, if $\D(\B)$ is taken as the domain of $\B$, then it is not clear that the limit martingale relation will hold for $\varphi_2$ (and any product in which $\varphi_2$ is a factor) because $\|\B_K\eta_K\varphi_2-\eta_K\B\varphi_2 \|$ does not converge to zero, as outlined above.  On the other hand, if $\D_0(\B)$ is taken as the domain of $\B$, then the martingale problem for $\B$ may have more than one solution.  For instance, if the initial distribution is the unit mass at $z=0$, then the identically zero stochastic process is a solution. 

We solve these problems by proving \textit{a priori} that, for any limit point $Z$ of $\{\rho_K(Z_K(\cdot))\}$, with probability one, $Z(t)\in\Ni$ for almost all $t\geq 0$. This is done in Lemma \ref{compcont} below. On $\Ni$, $\B\varphi_2$ can be approximated by $\B\varphi_{2+\eps}$, for $\eps\to 0+$, and this yields that the limit martingale relation, which holds for functions in $\D_0(\B)$, carries over to all functions in $\D(\B)$, and thus that the limit martingale problem has a unique solution (Theorem \ref{conv}). 

\begin{lemma}\label{relcomp}
$\{\rho_K(Z_K(\cdot))\}$ is relatively compact in $D_{\Nic}[0,\infty)$.
\end{lemma}

\begin{proof}
By Proposition \ref{rhoK-ZK}, $\rho_K(Z_K(\cdot))$ is a strong Markov process with generator the closure of $\B_K$ and sample paths in $C_{\NK}[0,\infty )$. Therefore $\rho_K(Z_K(\cdot))$ is a solution of the martingale problem for $\B_K$ (see, e.g., Proposition 4.1.7 in \cite{EK86}). 

We have, for $m,K\in \{2,3,\ldots\}$ and $z\in\NK$, 
\begin{align}\label{BKphi}
&\B_K\eta_K\varphi_m(z)\nonumber\\
&\quad{}=\binom{m}{2}(\varphi_{m-1}-\varphi_m)(z)\nonumber\\
&\qquad{}+\frac m2\bigg\{\frac {\theta +\alpha}{K-1}(\varphi_{m-1}-\varphi_m)(z)-\theta\varphi_m(z)\nonumber\\
&\qquad{}+\alpha\sum_{j=1}^K z_j r_j(z)\sum_{i=1}^Kp_i(z)z_i^{m-1}-\alpha\sum_{i=1}^K[z_i+z_i r_i(z)]z_i^{m-1}\bigg\}\nonumber\\
&\quad{}=\binom{m}{2}(\varphi_{m-1}-\varphi_m)(z)\nonumber\\
&\qquad{}+\frac m2\frac {\theta +\alpha}{K-1}(\varphi_{m-1}-\varphi_m)(z)-\frac m2(\theta\varphi_m+\alpha\varphi_{m-1})(z)\\
&\qquad{}+\frac m2\alpha\bigg\{\sum_{i=1}^K[1-z_i-z_i r_i(z)]z_i^{m-1}+\sum_{j=1}^K z_j r_j(z)\sum_{i=1}^Kp_i(z)z_i^{m-1} \bigg\}\nonumber\\
&\quad{}=\binom{m}{2}(\varphi_{m-1}-\varphi_m)(z)\nonumber\\
&\qquad{}+\frac m2\frac {\theta +\alpha}{K-1}(\varphi_{m-1}-\varphi_m)(z)-\frac m2(\theta\varphi_m+\alpha\varphi_{m-1})(z)\nonumber\\
&\qquad{}+\frac m2 \alpha\sum_{i=1}^K(1-z_i)^{K+1}z_i^{m-1}+\frac m2 \alpha\sum_{j=1}^K z_j r_j(z)\sum_{i=1}^Kp_i(z)z_i^{m-1}\nonumber\\
&\quad{}=\binom{m}{2}(\varphi_{m-1}-\varphi_m)(z)\nonumber\\
&\qquad{}+\frac m2\frac {\theta +\alpha}{K-1}(\varphi_{m-1}-\varphi_m)(z)-\frac m2(\theta\varphi_m+\alpha\varphi_{m-1})(z)\nonumber\\
&\qquad{}+\frac m2 \alpha\sum_{i=1}^K(1-z_i)^{K}z_i^{m-1}\bigg(1-z_i+\frac{\sum_{j=1}^K z_j r_j(z)}{\sum_{l=1}^K(1-z_l)^K}\bigg),\nonumber
\end{align}
where the third equality uses \eqref{condizioni r} and the fourth uses \eqref{condizioni p}.

Now, since $z_i\le1/i$ for $i=1,\ldots,K$, we have 
\begin{align*}
\sum_{l=1}^K(1-z_l)^K&\geq\sum_{l=\lfloor K/2\rfloor+1}^K(1-z_l)^K\\
&\geq\lceil K/2\rceil\bigg(1-\frac{1}{\lfloor K/2\rfloor+1}\bigg)^K\\
&\geq (K/2)e^{-2}, 
\end{align*}
so that 
\begin{equation}\label{fact2}
\frac {\sum\nolimits_{j=1}^K z_j r_j(z)}{\sum\nolimits_{l=1}^K(1-z_l)^K}\leq\frac{K}{(K/2)e^{-2}}=2e^2.
\end{equation}
In addition,  
\begin{equation}\label{fact1}
\sum_{i=1}^K(1-z_i)^Kz_i^{m-1}\leq K\sup_{0\leq u\leq 1}(1-u)^Ku^{m-1}\leq K\bigg(\frac{m-1}{K+m-1}\bigg)^{m-1}.
\end{equation}
Therefore, for each integer $m\geq 2$, 
\begin{equation}\label{ubdd-gen}
\sup_K\|\B_K\eta_K\varphi_m\|\leq C(\alpha ,\theta ,\varphi_m).
\end{equation}
For $\varphi ,\psi\in\D(\B)$, we can use the analogue of the first equation in (2.13) of \cite{EK81}, namely
\begin{equation}%\label{}
\B_K\eta_K(\varphi\psi )=(\eta_K\psi)\B_K\eta_K\varphi+(\eta_K\varphi)\B_K\eta_K\psi+\langle\grad(\eta_K\varphi),a\grad(\eta_K\psi)\rangle,\label{BKind}
\end{equation} 
where $a$ is given by \eqref{diff a}, to obtain 
\[
\|B_K\eta_K(\varphi\psi )\|\le\|\psi\| \|B_K\eta_K\varphi\|+\|\varphi\| \|B_K\eta_K\psi\|+2\sup_{i\ge1}\|\varphi_{z_i}\| \sup_{j\ge1}\|\psi_{z_j}\|.
\]
Then we can see, by induction on $l$, that \eqref{ubdd-gen} holds with $\varphi_m$ replaced by $\varphi$ of the form $\varphi =\varphi_{m_1}\varphi_{m_2}\cdots\varphi_{m_l}$, $m_1,\ldots ,m_l\in \{2,3,\ldots\}$, $l\in\N$, hence 
for every $\varphi\in \D(\B)$. 

Since $\D(\B)$ is dense in $C(\Nic)$, the lemma follows from Theorems 3.9.1 and 3.9.4 of \cite{EK86}. 
\end{proof}

\begin{lemma}\label{apriori}
For $2<m<3$ and $K\ge2$, 
\begin{align*}
\B_K\eta_K(\varphi_2-\varphi_m)
&\ge1-\alpha-\frac {m(m-1-\alpha )}2\varphi_{m-1}\\
&\quad{}-\bigg[(1+\theta )\varphi_2-\frac {m(m-1+\theta )}2\varphi_m\bigg]\\
&\quad{}-\bigg[\frac {3(\theta+\alpha)}{2(K-1)}+\frac {\alpha(1+2e^2)}{2(K+1)}\bigg]\quad\text{on }\nabla_K.
\end{align*} 
\end{lemma}

\begin{proof}
Let $2<m<3$ and $K\geq 2$. We have, on $\nabla_K$, 
\begin{align*}
&\B_K\eta_K(\varphi_2-\varphi_m)\\
&\quad{}=1-\alpha-\frac {m(m-1-\alpha )}2\varphi_{m-1}\\
&\qquad{}-\bigg[(1+\theta )\varphi_2-\frac {m(m-1+\theta )}2\varphi_m\bigg]\\
&\qquad{}+\frac {\theta +\alpha}{K-1}\bigg[1-\frac m2\varphi_{m-1}\bigg]-\frac {\theta +\alpha}{K-1}\bigg[\varphi_2-\frac m2\varphi_m\bigg]+\alpha R_{K,m},
\end{align*}
where
\begin{align*}
R_{K,m}(z)&:=\sum_{i=1}^K[1-z_i-z_i r_i(z)]z_i\Big(1-\frac m2 z_i^{m-2}\Big)\\
&\qquad{}+\sum_{j=1}^K z_j r_j(z)\sum_{i=1}^Kp_i(z)z_i\Big(1-\frac m2 z_i^{m-2}\Big).
\end{align*}
Since $z_i\le 1/i$ for $i=1,\ldots,K$, we obtain the inequalities  
\[
1-\frac m2z_i^{m-2}\geq 0,\quad i\geq 2,\qquad 1-\frac m2 z_1^{m-2}\geq -\frac 12,
\]
and hence
\[
R_{K,m}(z)\geq -\frac 12\bigg[(1-z_1)^{K+1}z_1+\sum_{j=1}^K z_j r_j(z)p_1(z)z_1\bigg].
\]
In addition, by \eqref{condizioni p} and \eqref{fact2}, 
\[
\sum_{j=1}^K z_j r_j(z)p_1(z)=\frac {\sum_{j=1}^K z_j r_j(z)}{\sum_{l=1}^K(1-z_l)^K}(1-z_1)^K\leq 2e^2(1-z_1)^K.
\]
Then, by the second inequality in \eqref{fact1}, we get $R_{K,m}(z)\geq -(1+2e^2)/(2(K+1))$.  Notice also that $1-\frac m2\varphi_{m-1}(z)\geq -\frac 12$ and that $\varphi_2(z)-\frac m2\varphi_m(z)\leq 1$.   The conclusion follows.
\end{proof}

\begin{lemma}\label{compcont}
For every limit point $Z$ of $\{\rho_K(Z_K(\cdot))\}$ in $D_{\Nic}[0,\infty )$, we have
\[
\e\bigg[\int_0^\infty\bigg(1-\sum_{i=1}^{\infty}Z_i(t)\bigg)\,dt\bigg]=0.
\]
\end{lemma}

\begin{proof}\ The proof is inspired by the first part of the proof of Theorem 2.6 in \cite{EK81}.  As $\rho_K(Z_K(\cdot))$ is a solution of the martingale problem for $\B_K$, Lemma \ref{apriori} implies that, for $2<m<3$ and $K\ge2$,
\begin{align}\label{Kmg-ineq}
\begin{split}
&\e\,\left[(\varphi_2-\varphi_m)(\rho_K(Z_K(T)))\right]\\
&\quad{}\geq\e\left[(\varphi_2-\varphi_m)(\rho_K(Z_K(0)))\right]\\
&\qquad{}+\e\left[\int_0^T\left(1-\alpha-\frac {m(m-1-\alpha )}2\varphi_{m-1}(\rho_K(Z_K(t)))\right) dt\right]\\
&\qquad{}-\e\left[\int_0^T\left((1+\theta )\varphi_2-\frac {m(m-1+\theta )}2\varphi_m\right)(\rho_K(Z_K(t)))\,dt\right]\\
&\;\qquad{}-\left[\frac {3(\theta +\alpha)}{2(K-1)}+\frac {\alpha (1+2e^2)}{2(K+1)}\right]T.
\end{split}
\end{align}
Let $Z$ be the limit in distribution of some subsequence $\{\rho_{K_{h}}(Z_{K_h})\}$. Since $\varphi_2$, $\varphi_m$, $\varphi_{m-1}$ are continuous and all 
integrands are bounded, by taking the limit as $h\to\infty$ along the subsequence $\{K_h\}$ in \eqref{Kmg-ineq}, we obtain 
\begin{eqnarray}\label{mg-ineq}
&&(1-\alpha )\e\bigg[\int_0^T\bigg(1-\frac {m(m-1-\alpha )}{2(1-\alpha)}\varphi_{m-1}(Z(t))\bigg)dt\bigg]\nonumber\\
&&\qquad{}\leq\e[(\varphi_2-\varphi_m)(Z(T))]-\e[(\varphi_2-\varphi_m)(Z(0))]\\
&&\qquad\qquad{}+(1+\theta )\e\bigg[\int_0^T\bigg(\varphi_2-\frac {m(m-1+\theta )}{2(1+\theta )}\varphi_m\bigg)(Z(t))\,dt\bigg].\nonumber
\end{eqnarray}
Since $\varphi_{m-1}(z)$ converges to $\sum_{i=1}^{\infty}z_i$ boundedly and pointwise on $\Nic$, we obtain the assertion by taking the limit as $m\to 2{+}$ in \eqref{mg-ineq}.
\end{proof}

\begin{lemma}\label{rate of conv}
For $0<\eps<1$, let $\varphi=\varphi_{m_1}\cdots\varphi_{m_l}$, where $m_1,\ldots,m_l\in[2+\eps,\infty)$.  Then
$$
\|\B_K\eta_K\varphi-\eta_K\B\varphi\|=O(K^{-\eps})\text{ as }K\to\infty.
$$
\end{lemma}

\begin{proof}
Consider first $\varphi =\varphi_m$ with $m\ge2+\eps$. Then  
\[
\B\varphi_m=\binom{m}{2}(\varphi_{m-1}-\varphi_m)-\frac m2(\theta\varphi_m+\alpha\varphi_{m-1}).
\]
Recalling \eqref{BKphi}--\eqref{fact1}, we have 
\begin{align*}
&\|\B_K\eta_K\varphi_m-\eta_K\B\varphi_m\|\\
&\quad{}\leq\frac m2\frac {\theta +\alpha}{K-1}\sup_{z\in\NK}|(\varphi_{m-1}-\varphi_m)(z)|\\
&\qquad{}+\sup_{z\in\nabla_K}\frac m2\alpha\sum_{i=1}^K(1-z_i)^K z_i^{m-1}\bigg(1-z_i+\frac{\sum\nolimits_{j=1}^K z_j r_j(z)}{\sum\nolimits_{l=1}^K (1-z_l)^K}\bigg)\\
&\quad\leq\frac m2\frac {\theta +\alpha}{K-1}+\frac m2\alpha K\bigg(\frac{m-1}{K+m-1}\bigg)^{m-1}(1+2e^2)
=O(K^{-\eps}),
\end{align*}
as required. 

By an analogue of the first equation in $(2.13)$ of \cite{EK81}, namely
\begin{equation}\label{Bind}
\B(\varphi\psi )=\psi\B\varphi+\varphi\B\psi+\langle\grad\varphi,a\grad\psi\rangle,
\end{equation}
we get, by \eqref{BKind},
\[
\|\B_{K}\eta_{K}(\varphi\psi )-\eta_{K}\B(\varphi\psi)\|\le\|\psi\|\|\B_{K}\eta_{K}\varphi-\eta_{K}\B\varphi\|+\|\varphi\|\|\B_{K}\eta_{K}\psi-\eta_{K}\B\psi\|.
\]
Thus, the statement of the lemma follows by induction on $l$.
\end{proof}

\begin{lemma}\label{conv on Ni}
Let $\varphi\in\D_0(\B)$ and $p\in\N$.  Then $\B(\varphi_{2+\eps}^p\varphi)\to\B(\varphi_2^p\varphi)$ boundedly and pointwise on $\Ni$.
\end{lemma}

\begin{proof}
By \eqref{Bind}, 
\begin{align*}
\B(\varphi_{2+\eps}^p\varphi)&=\varphi\B\varphi_{2+\eps}^p+\varphi_{2+\eps}^p\B\varphi+\langle\grad\varphi_{2+\eps}^p,a\grad\varphi\rangle\\
&=\varphi\bigg[p\varphi_{2+\eps}^{p-1}\B\varphi_{2+\eps}+\binom{p}2\varphi_{2+\eps}^{p-2}\langle\grad\varphi_{2+\eps},a\grad\varphi_{2+\eps}\rangle\bigg]\\
&\quad\;{}+\varphi_{2+\eps}^p\B\varphi+p\varphi_{2+\eps}^{p-1}\langle\grad\varphi_{2+\eps},a\grad\varphi\rangle\\
&\to\varphi\bigg[p\varphi_2^{p-1}\B\varphi_{2+}+\binom{p}2\varphi_2^{p-2}\langle\grad\varphi_2,a\grad\varphi_2\rangle\bigg]\\
&\quad\;{}+\varphi_2^p\B\varphi+p\varphi_2^{p-1}\langle\grad\varphi_2,a\grad\varphi\rangle
\end{align*}
boundedly and pointwise on $\Nic$ as $\eps$ goes to zero, where
\[
\B\varphi_{2+}(z):=\lim_{\eps\to 0}\B\varphi_{2+\eps}(z)=(1-\alpha )\sum_{
i=1}^{\infty}z_i-(1+\theta )\varphi_2(z),\quad z\in\Nic.
\]
We are also using
\begin{align*}
\langle\grad\varphi_{2+\eps},a\grad\varphi_{2+\eps}\rangle&=(2+\eps)^2(\varphi_{3+2\eps}-\varphi_{2+\eps}^2)\\
&\to4(\varphi_3-\varphi_2^2)=\langle\grad\varphi_2,a\grad\varphi_2\rangle
\end{align*}
and similarly $\langle\grad\varphi_{2+\eps},a\grad\varphi\rangle\to\langle\grad\varphi_2,a\grad\varphi\rangle$, both boundedly and pointwise on $\Nic$.  Of course,
\[
\B\varphi_2(z)=1-\alpha-(1+\theta )\varphi_2(z),\quad z\in\Nic,
\]
so $\B\varphi_{2+}=\B\varphi_2$ on $\Ni$.  We conclude that
\[
\B(\varphi_{2+\eps}^p\varphi)\to \varphi\B\varphi_2^p+\varphi_2^p\B\varphi+\langle\grad\varphi_2^p,a\grad\varphi\rangle=\B(\varphi_2^p\varphi)
\]
boundedly and pointwise on $\Ni$ (but not on $\Nic$).
\end{proof}

We are now ready to state our main result.

\begin{theorem}\label{conv}
Let $Z_K$ be the diffusion process of Proposition \ref{thm: K-diffusion} with initial distribution $\nu_K\in \mathscr{P}(\Delta_K)$.  Let $\B$ be given by \eqref{domain B}--\eqref{operator inf-inf} and let $Z$ be the diffusion process corresponding to the Feller semigroup generated by the closure in $C(\Nic)$ of $\B$, with initial distribution $\nu\in\mathscr{P}(\Nic)$. If $\nu_K\circ\rho_K^{-1}\Rightarrow\nu$, then 
\[
\rho_K(Z_K(\cdot))\Rightarrow Z(\cdot)\text{ in }C_{\Nic}[0,\infty).
\]
If in addition $\nu(\Ni)=1$, then the convergence holds in $C_{\Ni}[0,\infty )$.
\end{theorem}

\begin{proof}
First we prove convergence in $D_{\Nic}[0,\infty)$. The proof of this claim is in three steps:
\begin{itemize}
\item[(i)] Every limit point of $\{\rho_K(Z_K(\cdot))\}$ is a solution of the martingale problem for $\B$ as an operator on $\D_0(\B)$;
\item[(ii)] Every limit point of $\{\rho_K(Z_K(\cdot))\}$ is a solution of the martingale problem for $\B$ as an operator on $\D(\B)$;
\item[(iii)] The martingale problem for $\B$ as an operator on $\D(\B)$ has a unique solution for every initial distribution $\nu$.
\end{itemize}

\noindent \textit{Proof of \emph{(i)}.}
Let $Z$ be the limit in distribution of an arbitrary subsequence $\{\rho_{K_h}(Z_{K_h}(\cdot))\}$; see Lemma \ref{relcomp}. Since $\rho_{K_h}(Z_{K_h}(\cdot))$ is a solution of the 
martingale problem for $\B_{K_h}$, 
\[
\varphi(\rho_{K_h}(Z_{K_h}(t)))-\int_0^t\B_{K_h}\eta_{K_h}\varphi (\rho_{K_h}(Z_{K_h}(s)))\,ds=:M_{\varphi}^{(K_h)}(t)
\]
is a continuous martingale for every $\varphi\in \D_0(\B)$. 

By Lemma \ref{rate of conv}, $M_{\varphi}^{(K_h)}$ converges in distribution to 
\[
\varphi (Z(\cdot ))-\int_0^{\cdot}\B\varphi (Z(s))\,ds.
\]
On the other hand  $M_{\varphi}^{(K_h)}(t)$ is uniformly bounded for every $t$, hence the limit is a martingale.\smallskip

\noindent \textit{Proof of \emph{(ii)}.}
It is enough to prove that 
\[
(\varphi_2^p\varphi )(Z(\cdot ))-\int_0^{\cdot}\B(\varphi_2^p\varphi)(Z(s))\,ds=:M(\cdot)
\]
is a martingale for every $\varphi\in \D_0(\B)$ and every $p\in\N$. 

By Lemma \ref{compcont}, almost surely we have
\begin{equation}\label{zero}
\int_0^\infty(1-\I_{\Ni}(Z(s)))\,ds=0.
\end{equation}
Therefore, by Step 1, for every $\eps >0$, 
\[
(\varphi_{2+\eps}^p\varphi )(Z(t))-\int_0^t\I_{\Ni}(Z(s))\B(\varphi_{2+\eps}^p\varphi )(Z(s))\,ds=M_{\eps}(t)
\]
is a martingale. 

It follows from Lemma \ref{conv on Ni} that, almost surely, for all $t\geq 0$, $M_{\eps}(t)$ converges, to 
\[
(\varphi_2^p\varphi )(Z(t))-\int_0^t\I_{\Ni}(Z(s))\B(\varphi_2^p\varphi )(Z(s))\,ds,
\]
which in turn, by \eqref{zero}, almost surely, for all $t\geq 0$, equals $M(t)$.  On the other hand, for every $t\geq 0$, 
$M_{\eps}(t)$ is uniformly bounded, hence $M$ is a martingale.\smallskip 

\noindent \textit{Proof of \emph{(iii)}.}
A sufficient condition for uniqueness of the solution to the martingale problem for $\B$ is that, for each $\lambda>0$, ${\cal R}(\lambda I-\B)$, where ${\cal R}$ denotes the range and $I$ is the identity operator, is separating, i.e., such that, for any pair of probability measures $\mu,\nu\in\mathscr{P}(\Nic)$, $\int_{\Nic}f(z)\,\mu (dz)=\int_{\Nic}f(z)\,\nu (dz)$ for every $f\in {\cal R}(\lambda I-\B)$ implies $\mu =\nu$ (see, e.g., \cite{CK15}, Corollary 2.14).  In the present setup, since the closure of $\B$ generates a strongly continuous contraction semigroup on $C(\Nic)$ by \cite{P09}, then, for each $\lambda >0$, ${\cal R}(\lambda I-\B)$ is dense in $C(\Nic)$ (see, e.g., Proposition 1.2.1 in \cite{EK86}), therefore the condition is satisfied.\smallskip

Finally, the convergence holds in $C_{\Nic}[0,\infty )\subset D_{\Nic}[0,\infty )$ because the distributions of the processes $\rho_K(Z_K(\cdot))$ and $Z(\cdot)$ are concentrated on\linebreak $C_{\Nic}[0,\infty)$ and the Skorokhod topology relativised to $C_{\Nic}[0,\infty )$ coincides with the uniform-on-compact-sets topology on $C_{\Nic}[0,\infty )$ (see for example \cite{B68}, Section 18).  The last assertion of the theorem follows from \cite{E14} by the same argument. 
\end{proof}

\begin{remark} \label{gen-immig} A more careful inspection of the proofs shows that, if  the  mutation rates are given by \eqref{mutation rates}, all the results of this section hold for functions $p_i$ and $r_i$ satisfying the conditions of Section \ref{sec: WF chain} (in particular, $r_i(z)=r(z_i)$) and the following set of conditions:  For Lemma \ref{relcomp} we need only assume
$$
\sup_{z\in\NK}\sum_{j=1}^{K}z_j r_{j}(z)\sum_{i=1}^{K}p_{i}(z)z_{i}=O(1)\text{ as }K\to\infty.
$$
For Lemma \ref{apriori} with a possibly weaker but still adequate lower bound, it suffices that
$$
1-u-u r(u)\geq 0, \quad u\in[0,1], \qquad \sup_{u\in[0,1]}[1-u-ur(u)] u=o(1)\text{ as }K\to\infty,
$$
and
$$
\sup_{z\in\NK}\sum_{j=1}^{K}z_j r_{j}(z)p_1(z)z_1=o(1)\text{ as }K\to\infty.
$$
For Lemma \ref{rate of conv} with a possibly slower but still adequate rate of convergence, it is enough that
$$
\sup_{z\in\NK}\sum_{i=1}^{K}[1-z_i-z_ir(z_i)]z_i^{1+\eps}=o(1)\text{ as }K\to\infty,\quad 0<\eps<1,
$$
and 
$$
\sup_{z\in\NK}\sum_{j=1}^{K}z_jr_{j}(z)\sum_{i=1}^{K}p_{i}(z)z_{i}^{1+\eps}=o(1)\text{ as }K\to\infty,\quad 0<\eps<1.
$$
\end{remark}

%\%\%\%\%\%\%\%\%\%\%\%\%\%\%\%\%\%\%\%\%\%\%\%\%

\section{Convergence of stationary distributions}\label{sec: conv of stat dist}

We have seen that, for each $K\ge2$, $Z_K^N(\lfloor N\cdot\rfloor)\Rightarrow Z_K(\cdot)$ as $N\to\infty$ (Theorem \ref{thm: chain convergence}) and $\rho_K(Z_K(\cdot))\Rightarrow Z(\cdot)$ as $K\to\infty$ (Theorem \ref{conv}).  Now we want to obtain the analogous results for the stationary distributions.  Our Wright--Fisher Markov chain model is irreducible and aperiodic, and therefore has a unique stationary distribution $\mu_K^N\in\mathscr{P}(\Delta_K^N)$, which we regard as belonging to $\mathscr{P}(\Delta_K)$.  Our $K$-dimensional diffusion process $Z_K$ in $\Delta_K$ is ergodic by Theorem 3.2 of \cite{S81}, and therefore has a unique stationary distribution $\mu_K\in\mathscr{P}(\Delta_K)$.  Technically, Shiga's theorem does not apply to our model because, although our drift coefficients due to mutation meet his Condition II, our drift coefficients due to  migration,
$$
b_i(z):=\alpha p_i(z) \sum_{j=1}^{K}z_j r_{j}(z) -\alpha z_i r_{i}(z),
$$
are not of the form of his drift coefficients due to selection, 
$$
b_i(z):=z_i\bigg(\gamma_i(z)-\sum_{j=1}^Kz_j\gamma_j(z)\bigg).
$$  
Nevertheless, our drift coefficients due to migration do satisfy \eqref{boundary condition 1} and \eqref{boundary condition 2}, which together with smoothness is all that is needed for Shiga's proof. 
Finally, we denote by $\text{PD}(\theta,\alpha)\in\mathscr{P}(\Nic)$ the two-parameter Poisson--Dirichlet distribution, which is the unique stationary distribution of $Z$ in $\Nic$.  We will prove that, for each $K\ge2$, $\mu_K^N\Rightarrow\mu_K$ on $\DK$ as $N\to\infty$, and that $\mu_K\circ\rho_K^{-1}\Rightarrow\text{PD}(\theta,\alpha)$ on $\Nic$ as $K\to\infty$.  This is the two-parameter analogue of Kingman's result showing that the one-parameter Poisson--Dirichlet distribution $\text{PD}(\theta)$ is the weak limit of the descending order statistics of the symmetric Dirichlet distribution with parameter $\theta/(K-1)$.  It is not entirely analogous in that the symmetric Dirichlet distribution with parameter $\theta/(K-1)$ is much more explicit than $\mu_K$.  Nevertheless, it does allow us to give an interpretation to $\text{PD}(\theta,\alpha)$ in the context of population genetics.

\begin{theorem}
For each $K\ge2$, $\mu_K^N\Rightarrow\mu_K$ on $\Delta_K$ as $N\to\infty$.
\end{theorem}

\begin{proof}
For fixed $K\ge2$, $\{\mu_K^N\}$ is relatively compact because $\DK$ is compact.  It is enough to show that, if $\{N_m\}$ is a subsequence such that $\mu_K^{N_m}\Rightarrow\mu$ as $m\to\infty$, then $\mu=\mu_K$.  Given $f\in C^2(\Delta_K)$, 
$$
\int_{\Delta_K}{\cal A}_Kf\,d\mu=\lim_{m\to\infty}\int_{\Delta_K}{\cal A}_Kf\,d\mu_K^{N_m}=\lim_{m\to\infty}\int_{\Delta_K}N_m(\mathcal{T}_K^{N_m}-I)f\,d\mu_K^{N_m}=0,
$$
where the second equality uses \eqref{WF-conv}.  This shows that $\mu$ is the unique stationary distribution of $Z_K$, which we have denoted by $\mu_K$.
\end{proof}

\begin{theorem}
$\mu_K\circ\rho_K^{-1}\Rightarrow{\rm PD}(\theta,\alpha)$ on $\Nic$ as $K\to\infty$.
\end{theorem}

\begin{proof}
$\{\mu_K\circ\rho_K^{-1}\}$ is relatively compact because $\Nic$ is compact.  It is enough to show that, if $\{K_h\}$ is a subsequence such that $\mu_{K_h}\circ\rho_{K_h}^{-1}\Rightarrow\mu$ on $\Nic$ as $h\to\infty$, then $\mu=\text{PD}(\theta,\alpha)$.  First we show that $\mu$ is concentrated on $\Ni$.  It is intuitively clear and easy to prove that $\mu_K\circ\rho_K^{-1}$, which belongs to $\mathscr{P}(\NK)$ but can also be regarded as belonging to $\mathscr{P}(\Nic)$, is the unique stationary distribution of $\rho_K(Z_K(\cdot))$; indeed,
$$
\int_{\NK} \B_K f\,d(\mu_K\circ\rho_K^{-1})=\int_{\Delta_K}(\B_K f)\circ\rho_K\,d\mu_K
=\int_{\Delta_K}{\cal A}_K(f\circ\rho_K)\,d\mu_K=0,
$$
provided $f\in \D(\B_K)$.  Here we have used \eqref{A vs B}.  Lemma \ref{apriori} therefore implies that, for  $2<m<3$,
\begin{align*}
0=&\int_{\nabla_{K_h}}\B_{K_h}\eta_{K_h}(\varphi_2-\varphi_m)\,d(\mu_{K_h}\circ\rho_{K_h}^{-1})\\
\ge&\int_{\Nic}\bigg[\bigg(1-\alpha-\frac {m(m-1-\alpha )}2\varphi_{m-1}\bigg)\\
&\qquad\quad-\bigg((1+\theta )\varphi_2-\frac {m(m-1+\theta )}2\varphi_m\bigg)\bigg]\,d(\mu_{K_h}\circ\rho_{K_h}^{-1})\\
&-\left[\frac {3(\theta+\alpha)}{2(K_h-1)}+\frac {\alpha(1+2e^2)}{2(K_h+1)}\right]\\
\to&\int_{\Nic}\bigg[\bigg(1-\alpha-\frac {m(m-1-\alpha )}2\varphi_{m-1}\bigg)\\
&\qquad\quad-\bigg((1+\theta )\varphi_2-\frac {m(m-1+\theta )}2\varphi_m\bigg)\bigg]\,d\mu.
\end{align*}
In particular, this last integral is nonpositive.  Now let $m\to2+$ to conclude that
$$
(1-\alpha)\int_{\Nic}\bigg(1-\sum_{i=1}^\infty z_i\bigg)\,\mu(dz)\le0,
$$
or that $\mu(\Ni)=1$.

Next, from Lemma \ref{rate of conv} and $\mu_{K_h}\circ\rho_{K_h}^{-1}\Rightarrow\mu$ we get
\begin{align*}
0&=\,\lim_{h\to\infty}\int_{\nabla_{K_h}}\B_{K_h}\eta_{K_h}\varphi\,d(\mu_{K_h}\circ\rho_{K_h}^{-1})\\
&=\lim_{h\to\infty}\int_{\Nic}\B\varphi\,d(\mu_{K_h}\circ\rho_{K_h}^{-1})=\int_{\Nic}\B\varphi\,d\mu
\end{align*}
for all $\varphi\in\D_0(\B)$.  Let $\varphi\in\D_0(\B)$ and $p\in\N$.  Then from Lemma \ref{conv on Ni} we have
\begin{align*}
0&=\lim_{\eps\to0+}\int_{\Nic}\B(\varphi_{2+\eps}^p\varphi)\,d\mu=\lim_{\eps\to0+}\int_{\Ni}\B(\varphi_{2+\eps}^p\varphi)\,d\mu\\
&=\int_{\Ni}\B(\varphi_2^p\varphi)\,d\mu=\int_{\Nic}\B(\varphi_2^p\varphi)\,d\mu,
\end{align*}
implying that $\int_{\Nic}\B\varphi\,d\mu=0$ for every $\varphi\in\D(\B)$.  This tells us that $\mu=\text{PD}(\theta,\alpha)$, completing the proof.
\end{proof}

%\%\%\%\%\%\%\%\%\%\%\%\%\%\%\%\%\%\%\%\%\%\%\%\%

\section{The special case \texorpdfstring{$\theta\ge0$}{theta > 0}}
\label{sec:7}

The arguments of the previous sections assumed $0\le \alpha<1$ and $\theta>-\alpha$, which are the usual parameter constraints for $\PD(\theta,\alpha)$ distributions with nonnegative $\alpha$ and ensure that the mutation rate $\theta+\alpha$ in \eqref{mutation rates} is positive. It is interesting to note that if one imposes the stronger requirement that $\theta$ be nonnegative instead of $\theta>-\alpha$, then a modification of the construction allows us to separate the roles of $\theta$ and $\alpha$, which account for different mechanisms rather than jointly contributing to the mutation events. To this end, assume $\theta\ge0$ and modify \eqref{mutation rates} to 
\begin{equation}\label{theta mutation}
\bar u_{ij}:=\frac{\theta}{2N(K-1)}
\end{equation} 
and \eqref{condizioni r} to
\begin{equation*}\label{condizioni rbar}
\bar r_{i}(z):=\begin{cases}[1-(1-z_i)^{K}]/z_i&\text{if $z_i>0$}\\K&\text{if $z_i=0$.}\end{cases}
\end{equation*}
Accordingly, \eqref{drift b_i} becomes
\begin{equation*}
\bar b_i(z):=\,\frac{1}{2}\bigg[\frac{\theta}{K-1}(1-z_{i})-\theta z_{i}+\alpha p_i(z) \sum_{j=1}^{K} z_j \bar r_{j}(z) -\alpha z_i \bar r_{i}(z)\bigg].
\end{equation*} 
Similar arguments to those in the proof of Theorem \ref{conv} still hold in this setting, and there is an analogue of Remark \ref{gen-immig}. Now, however, $\theta$ alone is responsible for mutation through \eqref{theta mutation}, while $\alpha$ acts only through the migration mechanism. In contrast, the combined action of $\theta$ and $\alpha$ in \eqref{mutation rates} is partially remindful of the action of  the same parameters in Pitman's urn scheme construction of the $\PD(\theta,\alpha)$ distribution (see, for example, \cite{P95}, eq.~(15)), where $\theta$ and $\alpha$ jointly determine the probability of observing a new type in the sequence. 

%%%%%%%%%%%%%%%%%%%%%%%%%%%%%%%%%%%

\section*{Acknowledgments}

%The authors are grateful to two anonymous referees for carefully reading the paper and for providing helpful comments. 
The second and fourth authors are supported by the European Research Council (ERC) through StG ``N-BNP'' 306406. The third author is partially supported by a grant from the Simons Foundation (209632).

%%%%%%%%%%%%%%%%%%%%%%%%%%%%%%%%%%%%


\begin{thebibliography}{21}

\bibitem[Bertoin(2006)]{B06} {\sc Bertoin, J.} (2006). \emph{Random Fragmentation and Coagulation Processes}. Cambridge University Press, Cambridge.

\bibitem[Billingsley(1968)]{B68} {\sc Billingsley, P.} (1968). \emph{Convergence of Probability Measures}. Wiley, New York.

\bibitem[Costantini and Kurtz(2015)]{CK15}  {\sc Costantini, C.} and {\sc Kurtz, T.G.} (2015). Viscosity methods giving uniqueness for martingale problems. {\em Electron. J. Probab.} \textbf{20}, no.~67, 1--27.

\bibitem[Ethier(1976)]{E76}  {\sc Ethier, S. N.} (1976). A class of degenerate diffusion processes occurring in population genetics. {\em Comm. Pure Appl. Math.} \textbf{29}, 483--493.

\bibitem[Ethier(2014)]{E14}  {\sc Ethier, S. N.} (2014). A property of Petrov's diffusion. \emph{Electron. Comm. Probab.} \textbf{19}, no.~65, 1--4.

\bibitem[Ethier and Kurtz(1981)]{EK81}  {\sc Ethier, S. N.} and {\sc Kurtz, T. G.} (1981). The infinitely-many-neutral-alleles diffusion model. \emph{Adv. Appl. Probab.} \textbf{13}, 429--452.

\bibitem[Ethier and Kurtz(1986)]{EK86}  {\sc Ethier, S. N.} and {\sc Kurtz, T. G.} (1986). \emph{Markov Processes: Characterization and Convergence}. Wiley.

\bibitem[Ethier and Kurtz(1993)]{EK93}  {\sc Ethier, S. N.} and {\sc Kurtz, T. G.} (1993). Fleming--Viot processes in population genetics. \emph{SIAM J. Control Optim.} \textbf{31}, 345--386.
 
\bibitem[Feng(2010)]{F10}  {\sc Feng, S.} (2010). \emph{The Poisson--Dirichlet Distribution and Related Topics}. Springer, Heidelberg.

\bibitem[Feng and Sun(2010)]{FS10}  {\sc Feng, S.} and {\sc Sun, W.} (2010). Some diffusion processes associated with two parameter Poisson--Dirichlet distribution and Dirichlet process. \emph{Probab. Theory Relat. Fields} \textbf{148}, 501--525.

\bibitem[Feng, Sun, Wang and Xu(2011)]{Fetal11}  {\sc Feng, S., Sun, W., Wang, F.-Y.} and {\sc Xu, F.} (2011). Functional inequalities for the two-parameter extension of the infinitely-many-neutral-alleles diffusion. \emph{J. Funct. Anal.} \textbf{260}, 399--413.

\bibitem[Kingman(1975)]{K75}  {\sc Kingman, J. F. C.} (1975). Random discrete distributions. \emph{J. Roy. Statist. Soc. Ser. B} \textbf{37}, 1--22.

\bibitem[Lijoi and Pr\"unster(2009)]{LP09} {\sc Lijoi, A.} and {\sc Pr\"unster, I.} (2009). {Models beyond the Dirichlet process}. In Hjort, N. L., Holmes, C. C., M\"uller, P., Walker, S.G. (Eds.), \emph{Bayesian Nonparametrics}, Cambridge University Press.

\bibitem[Nagylaki(1992)]{N92}  {\sc Nagylaki, T.} (1992). \emph{Introduction of Theoretical Population Genetics}. Springer-Verlag, Berlin.

\bibitem[Perman, Pitman and Yor(1992)]{PPY92} {\sc Perman, M., Pitman, J.} and {\sc Yor, M.} (1992). Size-biased sampling of Poisson point processes and excursions. \emph{Probab. Theory Related Fields} \textbf{92}, 21--39.

\bibitem[Petrov(2009)]{P09}  {\sc Petrov, L.} (2009). Two-parameter family of diffusion processes in the Kingman simplex. \emph{Funct. Anal. Appl.} \textbf{43}, 279--296.

\bibitem[Pitman(1995)]{P95}  {\sc Pitman, J.} (1995). Exchangeable and partially exchangeable random partitions. \emph{Probab. Theory and Relat. Fields\/} {\bf 102}, 145--158.

\bibitem[Pitman(1996)]{P96}  {\sc Pitman, J.} (1996). Some developments of the Blackwell-MacQueen urn scheme. In \emph{Statistics, Probability and Game Theory} (T.S. Ferguson, L. S. Shapley and J.B. MacQueen, eds.). \emph{IMS Lecture Notes Monogr. Ser.} \textbf{30}, Inst. Math. Statist., Hayward, CA.

\bibitem[Pitman(2006)]{P06}  {\sc Pitman, J.} (2006). \emph{Combinatorial stochastic processes}. {\'Ecole d'\'et\'e de Probabilit\'es de Saint-Fleur XXXII. Lecture Notes in Math.} \textbf{1875}. Springer-Verlag.

\bibitem[Pitman and Yor(1997)]{PY97}  {\sc Pitman, J.} and {\sc Yor, M.} (1997). The two-parameter Poisson--Dirichlet distribution derived from a stable subordinator. \emph{Ann. Probab.} \textbf{25}, 855--900.

\bibitem[Ruggiero(2014)]{R14}  {\sc Ruggiero, M.} (2014). Species dynamics in the two-parameter Poisson--Dirichlet diffusion model. \emph{J. Appl. Probab.} \textbf{51}, 174--190.

\bibitem[Ruggiero and Walker(2009)]{RW09}  {\sc Ruggiero, M.} and {\sc Walker, S. G.} (2009). Countable representation for infinite-dimensional diffusions derived from the two-parameter Poisson--Dirichlet process. \emph{Electron. Comm. Probab.} \textbf{14}, 501--517.

\bibitem[Ruggiero, Walker and Favaro(2013)]{RWF13}  {\sc Ruggiero, M.,  Walker, S. G.} and {\sc Favaro, S.} (2013). Alpha-diversity processes and normalized inverse-Gaussian diffusions. \emph{Ann. Appl. Probab.} \textbf{23}, 386--425.

\bibitem[Sato(1978)]{S78}  {\sc Sato, K.} (1978). Diffusion operators in population genetics and convergence of Markov chains.  In \emph{Measure Theory Applications to Stochastic Analysis}, Lecture Notes in Mathematics \textbf{695}, 127--143.  Springer-Verlag, Berlin.

\bibitem[Shiga(1981)]{S81}  {\sc Shiga, T.} (1981). Diffusion processes in population genetics. \emph{J Math. Kyoto Univ.} \textbf{21} 133--151.

\bibitem[Teh and Jordan(2009)]{TJ09} {\sc Teh, Y. W.} and {\sc Jordan, M. I.} (2009). {Bayesian nonparametrics in machine learning}. In Hjort, N. L., Holmes, C. C., M\"uller, P., Walker, S. G. (Eds.), \emph{Bayesian Nonparametrics}, Cambridge University Press.

\bibitem[Zhou(2015)]{Z15} {\sc Zhou, Y.} (2015). {Ergodic inequality of a two-parameter infinitely-many-alleles diffusion model}.  \emph{J. Appl. Probab.} \textbf{52}, 238--246.

\end{thebibliography}
\end{document}